\documentclass[reqno]{amsart}

\usepackage{mathtools}
\usepackage{hyperref}
\usepackage{graphicx}

\hypersetup{colorlinks=true}

\theoremstyle{plain}
\newtheorem{theorem}{Theorem}[section]
\newtheorem{proposition}[theorem]{Proposition}
\newtheorem{lemma}[theorem]{Lemma}

\theoremstyle{definition}
\newtheorem*{problem}{Problem}

\numberwithin{equation}{section}

\title[Hilbert's inequality]{On the Montgomery--Vaughan weighted generalization of Hilbert's inequality}
\author{Wijit Yangjit}

\date{\today}

\address{Department of Mathematics, University of Michigan, 530 Church Street, Ann Arbor MI 48109--1043, USA}
\email{yangjit@umich.edu}

\begin{document}

\maketitle

\begin{abstract}
This paper concerns the problem of determining the optimal constant in the Montgomery--Vaughan weighted generalization of Hilbert's inequality. We consider an approach pursued by previous authors via a parametric family of inequalities. We obtain upper and lower bounds for the constants in inequalities in this family. A lower bound indicates that the method in its current form cannot achieve any value below $3.19497$, so cannot achieve the conjectured constant $\pi$. The problem of determining the optimal constant remains open.
\end{abstract}

\section{Introduction}

In this paper, we study a parametric family of inequalities, given in \eqref{general-alpha} below, that can yield an upper bound on the optimal constant in the Montgomery--Vaughan weighted generalization of Hilbert's inequality \eqref{MV2}. This inequality is important in the theory of the large sieve; see \cite{MV1973} and \cite{Mon1978}.

%
%

\subsection{History of the problem}

Let $N$ denote a positive integer, and let $z_1,\dots,z_N$ denote complex numbers. Hilbert's inequality states that
\begin{equation}\label{Hilbert}
\left\vert\sum_{m=1}^N\sum_{\substack{n=1\\n\neq m}}^N\frac{z_m\overline{z_n}}{m-n}\right\vert\le C_0\sum_{n=1}^N\left\vert z_n\right\vert^2,
\end{equation}
where $C_0$ is the absolute constant $2\pi$. Hilbert's proof was published by Weyl \cite[\S~15]{Wey1908}. In 1911, Schur \cite{Sch1911} obtained \eqref{Hilbert} with $C_0=\pi$ and demonstrated that this absolute constant is best possible. Hardy, Littlewood, and P\'{o}lya \cite[pp.~235--236]{HLP1952} gave an account of Hilbert's proof. Schur's proof is also reproduced in \cite[Theorem~294]{HLP1952}.

In 1974, Montgomery and Vaughan \cite{MV1974} established a generalization: If $\delta>0$ and $\left(\lambda_k\right)_{k=-\infty}^\infty$ is a sequence of real numbers such that $\lambda_{k+1}-\lambda_k\ge\delta$ for all $k$, then
\begin{equation}\label{MV1}
\left\vert\sum_{m=1}^N\sum_{\substack{n=1\\n\neq m}}^N\frac{z_m\overline{z_n}}{\lambda_m-\lambda_n}\right\vert\le\frac{\pi}{\delta}\sum_{n=1}^N\left\vert z_n\right\vert^2.
\end{equation}
Schur's bound is included in \eqref{MV1} as the case $\lambda_{k+1}-\lambda_k=\delta$. In the same paper, Montgomery and Vaughan also established a weighted form:
\begin{equation}\label{MV2}
\left\vert\sum_{m=1}^N\sum_{\substack{n=1\\n\neq m}}^N\frac{z_m\overline{z_n}}{\lambda_m-\lambda_n}\right\vert\le C_1\sum_{n=1}^N\frac{\left\vert z_n\right\vert^2}{\delta_n},
\end{equation}
where $\lambda_{k+1}>\lambda_k$ for all $k$ and $\delta_k\coloneqq\min\left\{\lambda_k-\lambda_{k-1},\lambda_{k+1}-\lambda_k\right\}$ and $C_1$ is the absolute constant $\frac{3\pi}{2}$. Denote by $\overline{C}_1$ the minimum of all absolute constants $C_1$ for which \eqref{MV2} holds. Montgomery and Vaughan \cite{MV1974} have raised the

\begin{problem}
Determine $\overline{C}_1$.
\end{problem}

By setting $\lambda_k=k$ in \eqref{MV2} and comparing with Schur's result, we see that
\begin{equation}\label{lowerC1}
\overline{C}_1\ge\pi.
\end{equation}
If $\overline{C}_1=\pi$, then \eqref{MV2} would contain \eqref{MV1}, and it is widely believed to be the case. In 1984, Preissmann \cite{Pre1984} proved that
\begin{equation}\label{upperC1}
\overline{C}_1\le\pi\sqrt{1+\frac{2}{3}\sqrt{\frac{6}{5}}}<\frac{4\pi}{3}.
\end{equation}
Preissmann's proof is based on that of Montgomery and Vaughan. Selberg (unpublished) said that he had shown that $\overline{C}_1\le3.2$ (which is $<\frac{54\pi}{53}$), but it seems that no trace remains of his argument; cf. \cite[p.~557]{Mon1978} and \cite[p.~145]{Mon1994}.

In 1981, Graham and Vaaler \cite{GV1981} constructed extreme majorants and minorants of the functions
$$
E(\beta,x)\coloneqq
\begin{cases}
e^{-\beta x}&\text{if }x\ge0,\\
0&\text{if }x<0,
\end{cases}
$$
where $\beta$ is an arbitrary positive real number, and used them to prove that
\begin{equation}\label{GrahamVaaler}
\frac{1}{\delta\left(e^{\beta/\delta}-1\right)}\sum_{n=1}^N\left\vert z_n\right\vert^2\le\sum_{m=1}^N\sum_{n=1}^N\frac{z_m\overline{z_n}}{\beta+2\pi i\left(\lambda_m-\lambda_n\right)}\le\frac{e^{\beta/\delta}}{\delta\left(e^{\beta/\delta}-1\right)}\sum_{n=1}^N\left\vert z_n\right\vert^2.
\end{equation}
The inequality \eqref{GrahamVaaler} includes \eqref{MV1} as the limiting case $\beta\rightarrow0^+$. In 1999, Montgomery and Vaaler \cite{MV1999} established a generalization of \eqref{MV2}:
\begin{equation}\label{MonVaa}
\left\vert\sum_{m=1}^N\sum_{\substack{n=1\\n\neq m}}^N\frac{z_m\overline{z_n}}{\beta_m+\beta_n+i\left(\lambda_m-\lambda_n\right)}\right\vert\le C_2\sum_{n=1}^N\frac{\left\vert z_n\right\vert^2}{\delta_n},
\end{equation}
where $\beta_1,\dots,\beta_N$ are nonnegative real numbers and $C_2$ is the absolute constant $84$, which is not optimal. Their proof involves the theory of $H^2$ functions in a half-plane and a maximal theorem of Hardy and Littlewood.

In 2005, Li \cite{Li2005} posed a question about the finite Hilbert transformation associated with a polynomial and proved that if the question always has an affirmative answer then $\overline{C}_1=\pi$.

%
%

\subsection{Main results}\label{subsectionC3}

We study the following parametric family of inequalities. For $0\le\alpha\le2$, let $\overline{C}(\alpha)$ be the minimum of all constants $C(\alpha)$ for which the inequality
\begin{equation}\label{general-alpha}
\sum_{m=1}^N\sum_{\substack{n=1\\n\neq m}}^N\frac{\delta_m^{2-\alpha}\delta_n^\alpha t_mt_n}{\left(\lambda_m-\lambda_n\right)^2}\le C(\alpha)\sum_{n=1}^Nt_n^2
\end{equation}
holds for all choices of a positive integer $N$, a strictly increasing sequence $\left(\lambda_k\right)_{k=-\infty}^\infty$ of real numbers,
$$
\delta_k\coloneqq\min\left\{\lambda_k-\lambda_{k-1},\lambda_{k+1}-\lambda_k\right\},
$$
and nonnegative real numbers $t_1,\dots,t_N$. Let $\overline{C}(\alpha)=\infty$ if there is no such real number $C(\alpha)$.

The value $\overline{C}\left(\frac{1}{2}\right)$ is relevant to the generalized Hilbert inequality \eqref{MV2}. In Section~\ref{C(1/2)theorems}, we shall prove the following inequality between $\overline{C}_1$ and $\overline{C}\left(\frac{1}{2}\right)$.
\begin{theorem}\label{C1C1/2inequality}
We have $\overline{C}_1\le\sqrt{\frac{\pi^2}{3}+2\overline{C}\left(\frac{1}{2}\right)}$.
\end{theorem}

The previous approaches to get an upper bound for $\overline{C}_1$ in \cite{MV1974}, \cite{Pre1984}, and \cite{Sha1984} rely on an upper bound for $\overline{C}\left(\frac{1}{2}\right)$ and Theorem \ref{C1C1/2inequality}. Montgomery and Vaughan \cite{MV1974} first showed that $\overline{C}\left(\frac{1}{2}\right)$ is finite. Specifically, they proved $\overline{C}\left(\frac{1}{2}\right)\le\frac{17}{2}$. The same bound has been used in \cite{MV1999} to prove \eqref{MonVaa}, but the best known upper bound for $\overline{C}\left(\frac{1}{2}\right)$ is due to Preissmann \cite{Pre1984}.

\begin{theorem}[Preissmann]\label{Preissmann'sTheorem}
We have $\overline{C}\left(\frac{1}{2}\right)\le\frac{\pi^2}{3}+\frac{\pi^2}{3}\sqrt{\frac{6}{5}}$.
\end{theorem}

By means of Theorem \ref{C1C1/2inequality}, Theorem \ref{Preissmann'sTheorem} implies \eqref{upperC1}. Another immediate consequence of Theorem \ref{C1C1/2inequality} is that \eqref{lowerC1} implies $\overline{C}\left(\frac{1}{2}\right)\ge\frac{\pi^2}{3}$. (This lower bound has been pointed out in \cite[p.~36]{MV1999}.) Moreover, the conjecture that $\overline{C}_1=\pi$ would follow if $\overline{C}\left(\frac{1}{2}\right)=\frac{\pi^2}{3}$.

In Section \ref{C(alpha)theorems}, we shall prove the following properties of $\overline{C}(\alpha)$.

\begin{theorem}\label{boundedness}
\emph{(1)} For real numbers $0\le\alpha\le2$, we have $\overline{C}(\alpha)=\overline{C}(2-\alpha)>0$.\\

\emph{(2)} For real numbers $0\le\alpha_1<\alpha_2\le2$ and $0<\theta<1$, we have
$$
\overline{C}\left(\theta\alpha_1+(1-\theta)\alpha_2\right)\le\overline{C}\left(\alpha_1\right)^\theta\overline{C}\left(\alpha_2\right)^{1-\theta}.
$$

\emph{(3)} For real numbers $0\le\alpha_1<\alpha_2\le1$, we have $\overline{C}\left(\alpha_1\right)\ge\overline{C}\left(\alpha_2\right)$. Therefore the minimum of $\overline{C}(\alpha)$ for $0\le\alpha\le2$ is attained at $\alpha=1$.\\

\emph{(4)} For real numbers $0\le\alpha<\frac{1}{2}$, we have $\overline{C}(\alpha)=\infty$.
\end{theorem}

Also in Section \ref{C(alpha)theorems}, we determine the minimum value.

\begin{theorem}\label{C(1)exact}
We have $\overline{C}(1)=\frac{\pi^2}{3}$.
\end{theorem}

In Section \ref{negative-section}, we shall prove a new lower bound for $\overline{C}\left(\frac{1}{2}\right)$.

\begin{theorem}\label{main-negative}
We have $\overline{C}\left(\frac{1}{2}\right)\ge0.35047\pi^2$.
\end{theorem}

From Theorem \ref{main-negative}, we deduce that any upper bound for $\overline{C}_1$ obtainable by Theorem \ref{C1C1/2inequality} cannot be smaller than $3.19497$. This method of using Theorem \ref{C1C1/2inequality} is incapable of proving $\overline{C}_1=\pi$.

\section{Preliminaries}\label{S}

%
%

\subsection{Eigenvalues of generalized weighted Hilbert matrices}

Let us consider $N\times N$ matrices $H=\left[h_{mn}\right]$ with entries given by
\begin{equation}\label{weighted-Hilbert-matrices}
h_{mn}\coloneqq
\begin{cases}
\frac{c_mc_n}{\lambda_m-\lambda_n}&\text{if }m\neq n,\\[.1in]
0&\text{if }m=n,
\end{cases}
\end{equation}
where $\left(\lambda_k\right)_{k=-\infty}^\infty$ is a strictly increasing sequence of real numbers and $c_1,\dots,c_N$ are positive real numbers. Since $H$ is skew-Hermitian (i.e., $iH$ is Hermitian), all its eigenvalues are purely imaginary. Let $\left[u_1,\dots,u_N\right]^\top$ be an eigenvector of $H$, and let $i\mu$ be its associated eigenvalue. That is,
$$
\sum_{\substack{n=1\\n\neq m}}^N\frac{c_mc_nu_n}{\lambda_m-\lambda_n}=i\mu u_m
$$
for all $m=1,\dots,N$.

It is well known (see, e.g., \cite[\S~7.4]{Mon1994}) that the numerical radius of a normal matrix is the same as its spectral radius (and its operator norm). Thus, if $i\mu$ has the largest modulus among all eigenvalues of $H$, then
\begin{equation}\label{inequality:numerical_radius}
\left\vert\sum_{m=1}^N\sum_{\substack{n=1\\n\neq m}}^N\frac{c_mc_nz_m\overline{z_n}}{\lambda_m-\lambda_n}\right\vert\le\lvert\mu\rvert\sum_{n=1}^N\left\vert z_n\right\vert^2
\end{equation}
for all complex numbers $z_1,\dots,z_N$. On replacing $z_n$ by $\frac{z_n}{c_n}$, we see that \eqref{inequality:numerical_radius} is equivalent to
\begin{equation}\label{inequality:c_n}
\left\vert\sum_{m=1}^N\sum_{\substack{n=1\\n\neq m}}^N\frac{z_m\overline{z_n}}{\lambda_m-\lambda_n}\right\vert\le\lvert\mu\rvert\sum_{n=1}^N\frac{\left\vert z_n\right\vert^2}{c_n^2}.
\end{equation}

One may obtain the generalized Hilbert inequality \eqref{MV2} with some constant $C_1$ from \eqref{inequality:c_n} by giving an upper bound for the sizes of eigenvalues of $H$ in the case that $c_n^2=\delta_n=\min\left\{\lambda_n-\lambda_{n-1},\lambda_{n+1}-\lambda_n\right\}$. A key result to that end is:
\begin{lemma}\label{identity:eigenv}
Let $\left[u_1,\dots,u_N\right]^\top$ be an eigenvector of $H$, and let $i\mu$ be its associated eigenvalue. Then the identity
\begin{equation}\label{Selberg}
\mu^2\left\vert u_m\right\vert^2=\sum_{\substack{n=1\\n\neq m}}^N\frac{c_m^2c_n^2\left\vert u_n\right\vert^2}{\left(\lambda_m-\lambda_n\right)^2}+2\sum_{\substack{n=1\\n\neq m}}^N\frac{c_m^3c_n\Re\left(\overline{u_m}u_n\right)}{\left(\lambda_m-\lambda_n\right)^2}
\end{equation}
holds for all $m=1,\dots,N$.
\end{lemma}
\begin{proof}
See Preissmann and L\'{e}v\^{e}que \cite[Lemma~5~(b)]{PL2013}.
\end{proof}

%
%

\subsection{A weighted spacing lemma and Shan's method}

The goal of this subsection is to prove:
\begin{lemma}\label{lemma:Preissmann}
Let $\left(\lambda_k\right)_{k=-\infty}^\infty$ be a strictly increasing sequence of real numbers. Denote by $\delta_k$ the minimum between $\lambda_k-\lambda_{k-1}$ and $\lambda_{k+1}-\lambda_k$. Then for real numbers $\sigma>1$ and integers $\ell$, we have
\begin{equation}\label{Preissmann}
\sum_{\substack{k=-\infty\\k\neq\ell}}^\infty\frac{\delta_k}{\left\vert\lambda_k-\lambda_\ell\right\vert^\sigma}\leq\frac{2\zeta(\sigma)}{\delta_\ell^{\sigma-1}}.
\end{equation}
\end{lemma}

One can show that equality holds in \eqref{Preissmann} if and only if the sequence $\left(\lambda_{k+1}-\lambda_k\right)_{k=-\infty}^\infty$ is constant, but we shall not treat it here.

Lemma \ref{lemma:Preissmann} is a direct consequence of Lemme 1 of Preissmann \cite{Pre1984}. We present a proof using a method of Shan \cite{Sha1984}, who independently derived Lemma \ref{lemma:Preissmann}. The work of Shan, done at the same time as that of Preissmann, is obscure and hard to obtain. Peng Gao (private communication) translated Shan's argument, which appears in \cite[pp.~590--595]{PP1991}. The next three lemmas are an exposition of Shan's method.

Let $f$ be a real-valued function, defined on the interval $[1,\infty)$. We will assume that $f$ satisfies some (or all) of the following four conditions:
\begin{enumerate}
\item[(a)]$f(\theta x+(1-\theta)y)\le \theta f(x)+(1-\theta)f(y)$ for all $0\le\theta\le1$ and $1\le x\le y$.
\item[(b)]$f(x)\ge f(y)$ for all $1\le x\le y$.
\item[(c)]$f(x)\ge0$ for all $x\ge1$.
\item[(d)]The series $\sum_{j=1}^\infty f(j)$ converges.
\end{enumerate}

We note that (c) follows from (b) and (d), since (b) implies $f(x)\ge\lim_{k\rightarrow\infty}f(k)$ and (d) implies $\lim_{k\rightarrow\infty}f(k)=0$.

\begin{lemma}\label{lemma:equidistance}
Assume that $f:[1,\infty)\rightarrow\mathbb{R}$ satisfies \emph{(a)} and \emph{(b)}. Let $\left(a_n\right)_{n=1}^\infty$ be a sequence of real numbers such that $a_n\ge1$ for all $n$. Set $\lambda_n\coloneqq\sum_{m=1}^na_m$. Then for positive integers $N$, we have
$$\sum_{n=1}^Na_nf\left(\lambda_n\right)\le\sum_{j=1}^{\left\lfloor\lambda_N\right\rfloor}f(j)+\left\{\lambda_N\right\}f\left(\left\lfloor\lambda_N\right\rfloor+1\right),$$
where $\{x\}=x-\lfloor x\rfloor$ denotes the fractional part of $x$.
\end{lemma}

\begin{proof}
By the convexity of $f$, we have
\begin{equation}\label{2}
f\left(\lambda_n\right)\le\left(1-\left\{\lambda_n\right\}\right)f\left(\left\lfloor\lambda_n\right\rfloor\right)+\left\{\lambda_n\right\}f\left(\left\lfloor\lambda_n\right\rfloor+1\right).
\end{equation}
Moreover, since $a_n\ge1$ and $f$ is weakly decreasing, it follows that
\begin{equation}\label{3}
\left(a_n-1\right)f\left(\lambda_n\right)\leq\left(a_n-1\right)f\left(\left\lfloor\lambda_n\right\rfloor\right).
\end{equation}
On summing \eqref{2} and \eqref{3}, we obtain
\begin{equation}\label{4}
a_nf\left(\lambda_n\right)\le\left(a_n-\left\{\lambda_n\right\}\right)f\left(\left\lfloor\lambda_n\right\rfloor\right)+\left\{\lambda_n\right\}f\left(\left\lfloor\lambda_n\right\rfloor+1\right).
\end{equation}
Now, we consider the first term on the right side of \eqref{4} and note that $\lambda_n=\lambda_{n-1}+a_n\ge\lambda_{n-1}+1$:
\begin{align*}
\left(a_n-\left\{\lambda_n\right\}\right)f\left(\left\lfloor\lambda_n\right\rfloor\right)&=\left(\left\lfloor\lambda_n\right\rfloor-\left\lfloor \lambda_{n-1}\right\rfloor-1\right)f\left(\left\lfloor\lambda_n\right\rfloor\right)+\left(1-\left\{\lambda_{n-1}\right\}\right)f\left(\left\lfloor\lambda_n\right\rfloor\right)\\
&\le\sum_{j=\left\lfloor\lambda_{n-1}\right\rfloor+2}^{\left\lfloor\lambda_n\right\rfloor}f(j)+\left(1-\left\{\lambda_{n-1}\right\}\right)f\left(\left\lfloor\lambda_{n-1}\right\rfloor+1\right)\\
&=\sum_{j=\left\lfloor\lambda_{n-1}\right\rfloor+1}^{\left\lfloor\lambda_n\right\rfloor}f(j)-\left\{\lambda_{n-1}\right\}f\left(\left\lfloor\lambda_{n-1}\right\rfloor+1\right).
\end{align*}
On inserting this in \eqref{4}, we get
\begin{equation}\label{5}
a_nf\left(\lambda_n\right)\le\sum_{j=\left\lfloor\lambda_{n-1}\right\rfloor+1}^{\left\lfloor\lambda_n\right\rfloor}f(j)-\left\{\lambda_{n-1}\right\}f\left(\left\lfloor\lambda_{n-1}\right\rfloor+1\right)+\left\{\lambda_n\right\}f\left(\left\lfloor\lambda_n\right\rfloor+1\right).
\end{equation}
The result follows by summing \eqref{5} over $n=1,\dots,N$; the resulting sum on the right side is a telescoping sum.
\end{proof}

In what follows, we consider
\begin{equation}
F_N(\mathbf{x})\coloneqq\sum_{n=1}^N\min\left\{x_n,x_{n+1}\right\}f\left(\sum_{m=1}^nx_m\right),
\end{equation}
where $\mathbf{x}=\left(x_n\right)_{n=1}^\infty$ is a sequence of positive real numbers with $x_1\ge1$.

\begin{lemma}\label{lemma:monotonicity}
Assume that $f:[1,\infty)\rightarrow\mathbb{R}$ satisfies \emph{(a)}--\emph{(c)}. Let $\mathbf{a}=\left(a_n\right)_{n=1}^\infty$ be a sequence of positive real numbers with $a_1\ge1$. Suppose that $\nu\ge2$ is an integer such that $a_{\nu-1}>a_\nu$. Let $0<\varepsilon\le a_{\nu-1}-a_\nu$. Define $\mathbf{b}=\left(b_n\right)_{n=1}^\infty$ by
$$
b_n\coloneqq
\begin{cases}
a_n&\text{for }n\neq\nu,\\
a_\nu+\varepsilon&\text{for }n=\nu.
\end{cases}
$$
Then for positive integers $N$, we have
\begin{equation}\label{F_Nmonovariant}
F_N(\mathbf{a})\le F_N(\mathbf{b}).
\end{equation}
\end{lemma}

\begin{proof}
If $N\le\nu-2$, then \eqref{F_Nmonovariant} is an identity. So let us assume that $N\ge\nu-1$. Put $\lambda_n\coloneqq\sum_{m=1}^na_m$. It follows from the definition of $b_n$ that
$$
\min\left\{b_n,b_{n+1}\right\}-\min\left\{a_n,a_{n+1}\right\}
\begin{cases}
=\varepsilon&\text{if }n=\nu-1,\\
\ge0&\text{if }n=\nu,\\
=0&\text{otherwise,}
\end{cases}
$$
$$
\sum_{m=1}^nb_m=
\begin{cases}
\lambda_n&\text{for }n\le\nu-1,\\
\lambda_n+\varepsilon&\text{for }n\ge\nu.
\end{cases}
$$
By the nonnegativity of $f$, $\min\left\{b_\nu,b_{\nu+1}\right\}f\left(\lambda_\nu+\varepsilon\right)\ge\min\left\{a_\nu,a_{\nu+1}\right\}f\left(\lambda_\nu+\varepsilon\right)$. So
\begin{equation}\label{F_n-diff}
F_N(\mathbf{b})-F_N(\mathbf{a})\ge\varepsilon f\left(\lambda_{\nu-1}\right)+\sum_{n=\nu}^N\min\left\{a_n,a_{n+1}\right\}\left(f\left(\lambda_n+\varepsilon\right)-f\left(\lambda_n\right)\right).
\end{equation}
By the convexity of $f$, it follows that
$$
\frac{f\left(\lambda_n+\varepsilon\right)-f\left(\lambda_n\right)}{\varepsilon}\ge\frac{f\left(\lambda_n\right)-f\left(\lambda_{n-1}\right)}{a_n}
$$
for all $n\ge2$. So \eqref{F_n-diff} implies that
\begin{align*}
F_N(\mathbf{b})-F_N(\mathbf{a})&\ge\varepsilon f\left(\lambda_{\nu-1}\right)+\varepsilon\sum_{n=\nu}^N\frac{\min\left\{a_n,a_{n+1}\right\}}{a_n}\left(f\left(\lambda_n\right)-f\left(\lambda_{n-1}\right)\right)\\
&\ge\varepsilon f\left(\lambda_{\nu-1}\right)+\varepsilon\sum_{n=\nu}^N\left(f\left(\lambda_n\right)-f\left(\lambda_{n-1}\right)\right)\\
&=\varepsilon f\left(\lambda_N\right)\ge0.
\end{align*}
Hence $F_N(\mathbf{a})\le F_N(\mathbf{b})$.
\end{proof}

We now prove an upper bound for $F_N(\mathbf{a})$ that depends only on $f$.

\begin{lemma}\label{lemma:F_N(a)UpperBound}
Assume that $f:[1,\infty)\rightarrow\mathbb{R}$ satisfies \emph{(a)}--\emph{(d)}. Let $\mathbf{a}=\left(a_n\right)_{n=1}^\infty$ be a sequence of positive real numbers with $a_1\ge1$. Then for positive integers $N$, we have
\begin{equation}\label{F_N(a)Bound}
F_N(\mathbf{a})\le\sum_{j=1}^\infty f(j).
\end{equation}
\end{lemma}

By taking $a_n=1$ for all $n$ and letting $N\rightarrow\infty$, we see that \eqref{F_N(a)Bound} is sharp.

\begin{proof}
Define a sequence $\mathbf{\overline{a}}=\left(\overline{a}_n\right)_{n=1}^\infty$ by $\overline{a}_n\coloneqq\max\left\{a_m:m=1,\dots,n\right\}$. Then $\overline{a}_{n+1}\ge\overline{a}_n$ for all $n$ and $\overline{a}_1=a_1\ge1$. Let $N$ be a positive integer. By applying Lemma \ref{lemma:monotonicity}, with $\varepsilon=a_{\nu-1}-a_\nu$, as many times as we need, we see that
\begin{equation}\label{apply-monotonicity}
F_N\left(\mathbf{a}\right)\le F_N\left(\mathbf{\overline{a}}\right)=\sum_{n=1}^N\overline{a}_nf\left(\overline{\lambda}_n\right),
\end{equation}
where $\overline{\lambda}_n\coloneqq\sum_{m=1}^n\overline{a}_m$.

By Lemma \ref{lemma:equidistance} and the nonnegativity of $f$, the right side of \eqref{apply-monotonicity} is
\begin{equation}\label{apply-equidistance}
\sum_{n=1}^N\overline{a}_nf\left(\overline{\lambda}_n\right)\le\sum_{j=1}^{\left\lfloor\overline{\lambda}_N\right\rfloor}f(j)+\left\{\overline{\lambda}_N\right\}f\left(\left\lfloor\overline{\lambda}_N\right\rfloor+1\right)\le\sum_{j=1}^\infty f(j).
\end{equation}
The result \eqref{F_N(a)Bound} follows by combining \eqref{apply-monotonicity} and \eqref{apply-equidistance}.
\end{proof}

We are now ready to prove Lemma \ref{lemma:Preissmann}.

\begin{proof}[Proof~of~Lemma~\emph{\ref{lemma:Preissmann}}]
Let $\ell$ be an integer. Define sequences $\mathbf{a}=\left(a_n\right)_{n=1}^\infty$ and $\mathbf{b}=\left(b_n\right)_{n=1}^\infty$ by
$$
a_n\coloneqq\frac{\lambda_{\ell+n}-\lambda_{\ell+n-1}}{\delta_\ell}\quad\text{and}\quad b_n\coloneqq\frac{\lambda_{\ell-n+1}-\lambda_{\ell-n}}{\delta_\ell},
$$
for all $n$. Then $\mathbf{a}$ and $\mathbf{b}$ are sequences of positive real numbers with
$$
a_1=\frac{\lambda_{\ell+1}-\lambda_\ell}{\delta_\ell}\ge1\quad\text{and}\quad b_1=\frac{\lambda_\ell-\lambda_{\ell-1}}{\delta_\ell}\ge1.
$$
We have
$$
\min\left\{a_n,a_{n+1}\right\}=\frac{\delta_{\ell+n}}{\delta_\ell}\quad\text{and}\quad\min\left\{b_n,b_{n+1}\right\}=\frac{\delta_{\ell-n}}{\delta_\ell},
$$
$$
\sum_{m=1}^na_m=\frac{\lambda_{\ell+n}-\lambda_\ell}{\delta_\ell}\quad\text{and}\quad\sum_{m=1}^nb_m=\frac{\lambda_\ell-\lambda_{\ell-n}}{\delta_\ell}.
$$
Let $\sigma>1$. Applying Lemma \ref{lemma:F_N(a)UpperBound} with $f(x)=\frac{1}{x^\sigma}$, we obtain
\begin{align*}
\delta_\ell^{\sigma-1}\sum_{\substack{k=\ell-N\\k\neq\ell}}^{\ell+N}\frac{\delta_k}{\left\vert\lambda_k-\lambda_\ell\right\vert^\sigma}&=\delta_\ell^{\sigma-1}\sum_{n=1}^N\left(\frac{\delta_{\ell+n}}{\left(\lambda_{\ell+n}-\lambda_\ell\right)^\sigma}+\frac{\delta_{\ell-n}}{\left(\lambda_\ell-\lambda_{\ell-n}\right)^\sigma}\right)\\
&=F_N(\mathbf{a})+F_N(\mathbf{b})\\
&\le2\sum_{j=1}^\infty f(j)=2\zeta(\sigma).
\end{align*}
The result \eqref{Preissmann} follows by letting $N\to\infty$.
\end{proof}

\section{Proofs of Theorems \ref{C1C1/2inequality} and \ref{Preissmann'sTheorem}}\label{C(1/2)theorems}

%
%

\subsection{Proof of Theorem \ref{C1C1/2inequality}}

\begin{proposition}\label{TwoForms}
Let $N$ be a positive integer. Let $\left(\lambda_k\right)_{k=-\infty}^\infty$ be a strictly increasing sequence of real numbers. Denote by $\delta_k$ the minimum between $\lambda_k-\lambda_{k-1}$ and $\lambda_{k+1}-\lambda_k$. Assume that $C_3$ is a positive constant such that the inequality
\begin{equation}\label{positive-symmetric}
\sum_{m=1}^N\sum_{\substack{n=1\\n\neq m}}^N\frac{\delta_m^\frac{3}{2}\delta_n^\frac{1}{2}t_mt_n}{\left(\lambda_m-\lambda_n\right)^2}\le C_3\sum_{n=1}^Nt_n^2
\end{equation}
holds for all nonnegative real numbers $t_1,\dots,t_N$. Then the inequality \eqref{MV2}  holds for all complex numbers $z_1,\dots,z_N$ with the constant $C_1=\sqrt{\frac{\pi^2}{3}+2C_3}$.
\end{proposition}

\begin{proof}
Suppose that \eqref{positive-symmetric} holds. Let $\left[u_1,\dots,u_N\right]^\top$ be a unit eigenvector of $H=\left[h_{mn}\right]$, where $h_{mn}$ are given by \eqref{weighted-Hilbert-matrices} with $c_n=\sqrt{\delta_n}$, and let $i\mu$ be the eigenvalue associated with this eigenvector. On applying Lemma \ref{identity:eigenv} and summing \eqref{Selberg} over $m$, we get
\begin{equation}\label{unweighted-identity}
\mu^2=\sum_{m=1}^N\sum_{\substack{n=1\\n\neq m}}^N\frac{\delta_m\delta_n\left\vert u_n\right\vert^2}{\left(\lambda_m-\lambda_n\right)^2}+2\sum_{m=1}^N\sum_{\substack{n=1\\n\neq m}}^N\frac{\delta_m^\frac{3}{2}\delta_n^\frac{1}{2}\Re\left(\overline{u_m}u_n\right)}{\left(\lambda_m-\lambda_n\right)^2}\le S+2T,
\end{equation}
where $S$ and $T$ are given by
$$
S\coloneqq\sum_{m=1}^N\sum_{\substack{n=1\\n\neq m}}^N\frac{\delta_m\delta_n\left\vert u_n\right\vert^2}{\left(\lambda_m-\lambda_n\right)^2}\quad\text{and}\quad T\coloneqq\sum_{m=1}^N\sum_{\substack{n=1\\n\neq m}}^N\frac{\delta_m^\frac{3}{2}\delta_n^\frac{1}{2}\left\vert u_m\right\vert\left\vert u_n\right\vert}{\left(\lambda_m-\lambda_n\right)^2}.
$$
On one hand, by Lemma \ref{lemma:Preissmann}, we obtain
\begin{equation}\label{S-bound}
S=\sum_{n=1}^N\delta_n\left\vert u_n\right\vert^2\left(\sum_{\substack{m=1\\m\neq n}}^N\frac{\delta_m}{\left(\lambda_m-\lambda_n\right)^2}\right)\le\sum_{n=1}^N\delta_n\left\vert u_n\right\vert^2\left(\frac{\pi^2}{3\delta_n}\right)=\frac{\pi^2}{3}.
\end{equation}
On the other hand, substituting $t_n=\left\vert u_n\right\vert$ in \eqref{positive-symmetric} gives
\begin{equation}\label{T-bound}
T\le C_3.
\end{equation}
It follows from \eqref{unweighted-identity}, \eqref{S-bound}, and \eqref{T-bound} that
\begin{equation}\label{mu-bound1}
\vert\mu\vert\le\sqrt{S+2T}\le\sqrt{\frac{\pi^2}{3}+2C_3}.
\end{equation}
By the argument preceding \eqref{inequality:c_n}, we deduce from \eqref{inequality:c_n} and \eqref{mu-bound1} that \eqref{MV2} holds with $C_1=\sqrt{\frac{\pi^2}{3}+2C_3}$.
\end{proof}

One weak point in the proof of Proposition \ref{TwoForms} is the bound in \eqref{unweighted-identity}, where we disregard cancellation between terms.

\begin{proof}[Proof~of~Theorem~\emph{\ref{C1C1/2inequality}}]
Since \eqref{positive-symmetric} holds with $C_3=\overline{C}\left(\frac{1}{2}\right)$, it follows by Proposition \ref{TwoForms} that \eqref{MV2} holds with $C_1=\sqrt{\frac{\pi^2}{3}+2\overline{C}\left(\frac{1}{2}\right)}$. Hence the result follows.
\end{proof}

%
%

\subsection{Proof of Theorem \ref{Preissmann'sTheorem}}

\begin{lemma}\label{PreissmannLemme6}
Let $\left(\lambda_k\right)_{k=-\infty}^\infty$ be a strictly increasing sequence of real numbers. Denote by $\delta_k$ the minimum between $\lambda_k-\lambda_{k-1}$ and $\lambda_{k+1}-\lambda_k$. Then for distinct integers $\ell$ and $m$, we have
\begin{equation}
\sum_{\substack{k=-\infty\\k\neq\ell\\k\neq m}}^\infty\frac{\delta_k}{\left(\lambda_k-\lambda_\ell\right)^2\left(\lambda_k-\lambda_m\right)^2}\le\frac{\pi^2\left(\delta_\ell+\delta_m\right)}{3\delta_\ell\delta_m\left(\lambda_\ell-\lambda_m\right)^2}-\frac{3\left(\delta_\ell+\delta_m\right)}{\left(\lambda_\ell-\lambda_m\right)^4}.
\end{equation}
\end{lemma}
\begin{proof}
See Preissmann \cite[Lemme~6]{Pre1984}.
\end{proof}

\begin{proof}[Proof~of~Theorem~\emph{\ref{Preissmann'sTheorem}}]
Let
$$
U\coloneqq\sum_{m=1}^N\sum_{\substack{n=1\\n\neq m}}^N\frac{\delta_m^\frac{3}{2}\delta_n^\frac{1}{2}t_mt_n}{\left(\lambda_m-\lambda_n\right)^2}\quad\text{and}\quad V\coloneqq\sum_{n=1}^Nt_n^2.
$$
By Cauchy's inequality,
\begin{align*}
U^2&=\left(\sum_{n=1}^Nt_n\sum_{\substack{m=1\\m\neq n}}^N\frac{\delta_m^\frac{3}{2}\delta_n^\frac{1}{2}t_m}{\left(\lambda_m-\lambda_n\right)^2}\right)^2\\
&\le\left(\sum_{n=1}^Nt_n^2\right)\left(\sum_{n=1}^N\left(\sum_{\substack{m=1\\m\neq n}}^N\frac{\delta_m^\frac{3}{2}\delta_n^\frac{1}{2}t_m}{\left(\lambda_m-\lambda_n\right)^2}\right)^2\right)=V(S+T),
\end{align*}
where
$$
S\coloneqq\sum_{n=1}^N\sum_{\substack{m=1\\m\neq n}}^N\frac{\delta_m^3\delta_nt_m^2}{\left(\lambda_m-\lambda_n\right)^4}\quad\text{and}\quad T\coloneqq\sum_{n=1}^N\sum_{\substack{\ell=1\\\ell\neq n}}^N\sum_{\substack{m=1\\m\neq n\\m\neq\ell}}^N\frac{\delta_\ell^\frac{3}{2}\delta_m^\frac{3}{2}\delta_nt_\ell t_m}{\left(\lambda_\ell-\lambda_n\right)^2\left(\lambda_m-\lambda_n\right)^2}.
$$
Applying Lemma \ref{lemma:Preissmann} with $\sigma=4$, we obtain
$$
S=\sum_{m=1}^N\delta_m^3t_m^2\left(\sum_{\substack{n=1\\n\neq m}}^N\frac{\delta_n}{\left(\lambda_n-\lambda_m\right)^4}\right)\le\sum_{m=1}^N\delta_m^3t_m^2\left(\frac{\pi^4}{45\delta_m^3}\right)=\frac{\pi^4}{45}V.
$$
Applying Lemma \ref{PreissmannLemme6}, we obtain
\begin{align*}
T&=\sum_{\ell=1}^N\sum_{\substack{m=1\\m\neq\ell}}^N\delta_\ell^\frac{3}{2}\delta_m^\frac{3}{2}t_\ell t_m\left(\sum_{\substack{n=1\\n\neq\ell\\n\neq m}}^N\frac{\delta_n}{\left(\lambda_n-\lambda_\ell\right)^2\left(\lambda_n-\lambda_m\right)^2}\right)\\
&\le\sum_{\ell=1}^N\sum_{\substack{m=1\\m\neq\ell}}^N\delta_\ell^\frac{3}{2}\delta_m^\frac{3}{2}t_\ell t_m\left(\frac{\pi^2\left(\delta_\ell+\delta_m\right)}{3\delta_\ell\delta_m\left(\lambda_\ell-\lambda_m\right)^2}\right)=\frac{2\pi^2}{3}U.
\end{align*}
So $U^2\le V\left(\frac{\pi^4}{45}V+\frac{2\pi^2}{3}U\right)$. Solving this gives $U\le\left(\frac{\pi^2}{3}+\frac{\pi^2}{3}\sqrt{\frac{6}{5}}\right)V$.
\end{proof}

\section{Proofs of Theorems \ref{boundedness} and \ref{C(1)exact}}\label{C(alpha)theorems}

%
%

\subsection{Proof of Theorem \ref{boundedness}}

For real numbers $0\le\alpha\le2$ and positive integers $N$, let $\overline{C}(\alpha,N)$ be the minimum of all constants $C(\alpha,N)$ for which the inequality
\begin{equation}\label{general-alpha-N}
\sum_{m=1}^N\sum_{\substack{n=1\\n\neq m}}^N\frac{\delta_m^{2-\alpha}\delta_n^\alpha t_mt_n}{\left(\lambda_m-\lambda_n\right)^2}\le C(\alpha,N)\sum_{n=1}^Nt_n^2
\end{equation}
holds for all choices of a strictly increasing sequence $\left(\lambda_k\right)_{k=-\infty}^\infty$ of real numbers,
$$
\delta_k\coloneqq\min\left\{\lambda_k-\lambda_{k-1},\lambda_{k+1}-\lambda_k\right\},
$$
and nonnegative real numbers $t_1,\dots,t_N$.

\begin{proposition}\label{C(alpha,N)basic}
\emph{(1)} For real numbers $0\le\alpha\le2$, we have $\overline{C}(\alpha,1)=0$ and $\overline{C}(\alpha,2)=1$.\\

\emph{(2)} For real numbers $0\le\alpha\le2$ and positive integers $N$, we have $\overline{C}(\alpha,N)\le\overline{C}(\alpha,N+1)$.\\

\emph{(3)} For real numbers $0\le\alpha\le2$ and positive integers $N$, we have $0\le\overline{C}(\alpha,N)\le N-1$.\\

\emph{(4)} For real numbers $0\le\alpha\le2$, we have $\overline{C}(\alpha)=\lim_{N\rightarrow\infty}\overline{C}(\alpha,N)$.
\end{proposition}

\begin{proof}
(1) If $N=1$, the left side of \eqref{general-alpha-N} is $0$. So $\overline{C}(\alpha,1)=0$. If $N=2$, the left side of \eqref{general-alpha-N} is $2t_1t_2$. So $\overline{C}(\alpha,2)=1$.

(2) Let $t_1,\dots,t_N$ be nonnegative real numbers, and let $t_{N+1}=0$. Then
$$
\sum_{m=1}^N\sum_{\substack{n=1\\n\neq m}}^N\frac{\delta_m^{2-\alpha}\delta_n^\alpha t_mt_n}{\left(\lambda_m-\lambda_n\right)^2}=\sum_{m=1}^{N+1}\sum_{\substack{n=1\\n\neq m}}^{N+1}\frac{\delta_m^{2-\alpha}\delta_n^\alpha t_mt_n}{\left(\lambda_m-\lambda_n\right)^2}\le\overline{C}(\alpha,N+1)\sum_{n=1}^{N+1}t_n^2=\overline{C}(\alpha,N+1)\sum_{n=1}^Nt_n^2.
$$
So $\overline{C}(\alpha,N)\le\overline{C}(\alpha,N+1)$.

(3) We have
$$
\sum_{m=1}^N\sum_{\substack{n=1\\n\neq m}}^N\frac{\delta_m^{2-\alpha}\delta_n^\alpha t_mt_n}{\left(\lambda_m-\lambda_n\right)^2}\le\sum_{m=1}^N\sum_{\substack{n=1\\n\neq m}}^Nt_mt_n\le\sum_{m=1}^N\sum_{\substack{n=1\\n\neq m}}^N\frac{t_m^2+t_n^2}{2}=(N-1)\sum_{n=1}^Nt_n^2.
$$
So $\overline{C}(\alpha,N)\le N-1$. On the other hand, from (2) and (1), we have $\overline{C}(\alpha,N)\ge\overline{C}(\alpha,1)=0$.

(4) Since \eqref{general-alpha-N} holds with $C(\alpha,N)=\overline{C}(\alpha)$, it follows that $\overline{C}(\alpha,N)\le\overline{C}(\alpha)$ for all $N$. Hence $\lim_{N\rightarrow\infty}\overline{C}(\alpha,N)\le\overline{C}(\alpha)$. On the other hand, by (2), $\lim_{N\rightarrow\infty}\overline{C}(\alpha,N)=\sup_N\overline{C}(\alpha,N)$. So \eqref{general-alpha} holds with $C(\alpha)=\lim_{N\rightarrow\infty}\overline{C}(\alpha,N)$. Hence $\overline{C}(\alpha)\le\lim_{N\rightarrow\infty}\overline{C}(\alpha,N)$.
\end{proof}

\begin{proposition}\label{proposition:C(alpha,N)}
\emph{(1)} For real numbers $0\le\alpha\le2$ and integers $N\ge2$, we have $\overline{C}(\alpha,N)=\overline{C}(2-\alpha,N)\ge1$.\\

\emph{(2)} For real numbers $0\le\alpha_1<\alpha_2\le2$ and $0<\theta<1$, and for positive integers $N$, we have
$$
\overline{C}\left(\theta\alpha_1+(1-\theta)\alpha_2,N\right)\le\overline{C}\left(\alpha_1,N\right)^\theta\overline{C}\left(\alpha_2,N\right)^{1-\theta}.
$$

\emph{(3)} For real numbers $0\le\alpha_1<\alpha_2\le1$ and positive integers $N$, we have $\overline{C}\left(\alpha_1,N\right)\ge\overline{C}\left(\alpha_2,N\right)$.\\

\emph{(4)} For real numbers $0\le\alpha<\frac{1}{2}$ and integers $N\ge2$, we have $\overline{C}(\alpha,N)\gg N^{\frac{1}{2}-\alpha}$.
\end{proposition}

\begin{proof}
(1) The left side of \eqref{general-alpha-N} is unchanged on replacing $\alpha$ by $2-\alpha$. It follows that $\overline{C}(\alpha,N)=\overline{C}(2-\alpha,N)$. In addition, by Proposition \ref{C(alpha,N)basic}, we see that $\overline{C}(\alpha,N)\ge\overline{C}(\alpha,2)=1$.

(2) Let $\alpha=\theta\alpha_1+(1-\theta)\alpha_2$. Apply H\"{o}lder's inequality:
\begin{align*}
\sum_{m=1}^N\sum_{\substack{n=1\\n\neq m}}^N\frac{\delta_m^{2-\alpha}\delta_n^\alpha t_mt_n}{\left(\lambda_m-\lambda_n\right)^2}&\le\left(\sum_{m=1}^N\sum_{\substack{n=1\\n\neq m}}^N\frac{\delta_m^{2-\alpha_1}\delta_n^{\alpha_1}t_mt_n}{\left(\lambda_m-\lambda_n\right)^2}\right)^\theta\left(\sum_{m=1}^N\sum_{\substack{n=1\\n\neq m}}^N\frac{\delta_m^{2-\alpha_2}\delta_n^{\alpha_2}t_mt_n}{\left(\lambda_m-\lambda_n\right)^2}\right)^{1-\theta}\\
&\le\overline{C}\left(\alpha_1,N\right)^\theta\overline{C}\left(\alpha_2,N\right)^{1-\theta}\sum_{n=1}^Nt_n^2.
\end{align*}
So $\overline{C}(\alpha,N)\le\overline{C}\left(\alpha_1,N\right)^\theta\overline{C}\left(\alpha_2,N\right)^{1-\theta}$.

(3) Let $\theta=\frac{2-\alpha_1-\alpha_2}{2\left(1-\alpha_1\right)}$. Then $0<\theta<1$ and $\alpha_2=\theta\alpha_1+\left(1-\theta\right)\left(2-\alpha_1\right)$. By (2), we have
$$
\overline{C}\left(\alpha_2,N\right)=\overline{C}\left(\theta\alpha_1+(1-\theta)\left(2-\alpha_1\right),N\right)\le\overline{C}\left(\alpha_1,N\right)^\theta\overline{C}\left(2-\alpha_1,N\right)^{1-\theta}.
$$
The last quantity is equal to $\overline{C}\left(\alpha_1,N\right)$ by (1).

(4) We choose $\lambda_k=k$ for $k\le1$ and $\lambda_{2+\ell}=2+\frac{\ell}{N}$ for $\ell\ge0$. Then $\delta_k=1$ for $k\le1$ and $\delta_{2+\ell}=\frac{1}{N}$ for $\ell\ge0$. Choose $t_1=\sqrt{\frac{N+1}{2N}}$ and $t_n=\frac{1}{\sqrt{2N}}$ for $2\le n\le N$. So $\sum_{n=1}^Nt_n^2=1$, and \eqref{general-alpha-N} yields
$$
C(\alpha,N)\ge\sum_{m=1}^N\sum_{\substack{n=1\\n\neq m}}^N\frac{\delta_m^{2-\alpha}\delta_n^\alpha t_mt_n}{\left(\lambda_m-\lambda_n\right)^2}\ge\sum_{n=2}^N\frac{\delta_1^{2-\alpha}\delta_n^\alpha t_1t_n}{\left(\lambda_1-\lambda_n\right)^2}=\sum_{n=2}^N\frac{\sqrt{N+1}}{2N^{\alpha+1}\left(1+\frac{n-2}{N}\right)^2}.
$$
The last quantity is $\gg N^{\frac{1}{2}-\alpha}$ for $N\ge2$. Hence $\overline{C}(\alpha,N)\gg N^{\frac{1}{2}-\alpha}$ for $N\ge2$.
\end{proof}

\begin{proof}[Proof~of~Theorem~\emph{\ref{boundedness}}]
The result follows as we let $N\rightarrow\infty$ in Proposition \ref{proposition:C(alpha,N)}.
\end{proof}

%
%

\subsection{Proof of Theorem \ref{C(1)exact}}

\begin{proposition}\label{pi^2/3}
Let $\left(\lambda_k\right)_{k=-\infty}^\infty$ be a strictly increasing sequence of real numbers. Denote by $\delta_k$ the minimum between $\lambda_k-\lambda_{k-1}$ and $\lambda_{k+1}-\lambda_k$. Then for any sequence $\left(t_1,\dots,t_N\right)$ of nonnegative real numbers,
\begin{equation}\label{positive-symmetric-weaker}
\sum_{m=1}^N\sum_{\substack{n=1\\n\neq m}}^N\frac{\delta_m\delta_nt_mt_n}{\left(\lambda_m-\lambda_n\right)^2}\le\frac{\pi^2}{3}\sum_{n=1}^Nt_n^2.
\end{equation}
\end{proposition}
\begin{proof}
By the arithmetic--geometric mean inequality,
$$
\sum_{m=1}^N\sum_{\substack{n=1\\n\neq m}}^N\frac{\delta_m\delta_nt_mt_n}{\left(\lambda_m-\lambda_n\right)^2}\le\sum_{m=1}^N\sum_{\substack{n=1\\n\neq m}}^N\frac{\delta_m\delta_n\left(t_m^2+t_n^2\right)}{2\left(\lambda_m-\lambda_n\right)^2}=\sum_{n=1}^N\delta_nt_n^2\left(\sum_{\substack{m=1\\m\neq n}}^N\frac{\delta_m}{\left(\lambda_m-\lambda_n\right)^2}\right).
$$
By Lemma \ref{lemma:Preissmann}, the right side above is $\le\sum_{n=1}^N\delta_nt_n^2\left(\frac{\pi^2}{3\delta_n}\right)=\frac{\pi^2}{3}\sum_{n=1}^Nt_n^2$.
\end{proof}

\begin{proof}[Proof~of~Theorem~\emph{\ref{C(1)exact}}]
Proposition \ref{pi^2/3} shows $\overline{C}(1)\le\frac{\pi^2}{3}$. Now taking $\lambda_n=n$ and $t_n=\frac{1}{\sqrt{N}}$ in \eqref{general-alpha-N} yields
$$
\overline{C}(\alpha,N)\ge\frac{2}{N}\sum_{n=1}^{N-1}\frac{N-n}{n^2}=2\sum_{n=1}^{N-1}\frac{1}{n^2}-\frac{2}{N}\sum_{n=1}^{N-1}\frac{1}{n}.
$$
Letting $N\rightarrow\infty$ gives $\overline{C}(\alpha)\ge\frac{\pi^2}{3}$ for all $0\le\alpha\le2$. Hence $\overline{C}(1)=\frac{\pi^2}{3}$.
\end{proof}

%
%

\section{Proof of Theorem \ref{main-negative}}\label{negative-section}

Let $M$ denote a positive integer, and let $x_1,\dots,x_M$ denote real numbers, distinct modulo $1$. Put
$$
d_m\coloneqq\min_{n\neq m}\left\lVert x_n-x_m\right\rVert,
$$
where $\lVert x\rVert=\min_{k\in\mathbb{Z}}\lvert x-k\rvert$ denotes the distance between $x$ and a nearest integer. In the case that $M=1$, we let $d_1\coloneqq1$. Let $\tau_1,\dots,\tau_M$ denote nonnegative real numbers.

\begin{lemma}\label{lemma:trigonometric-equivalence}
The inequality \eqref{positive-symmetric} holds (for all $N$, $\lambda_n$, $\delta_n$, and $t_n$) if and only if the inequality
\begin{equation}\label{trigonometric}
\frac{1}{3}\sum_{m=1}^Md_m^2\tau_m^2+\sum_{m=1}^M\sum_{\substack{n=1\\n\neq m}}^M\frac{d_m^\frac{3}{2}d_n^\frac{1}{2}\tau_m\tau_n}{\sin^2\left(\pi\left(x_m-x_n\right)\right)}\le\frac{C_3}{\pi^2}\sum_{m=1}^M\tau_m^2
\end{equation}
holds for all positive integer $M$, distinct real numbers $x_1,\dots,x_M$ modulo $1$,
\begin{equation}\label{definition:d_m}
d_m\coloneqq\min\left\{\left\vert x_n-x_m-k\right\vert:k\in\mathbb{Z}\right\}\backslash\{0\},
\end{equation}
and nonnegative real numbers $\tau_1,\dots,\tau_M$.
\end{lemma}
\begin{proof}
($\Rightarrow$) Suppose that \eqref{positive-symmetric} holds. Let $x_1,\dots,x_M$ be real numbers, distinct modulo $1$. By symmetry in $x_1,\dots,x_M$, we may assume without loss of generality that $x_1<\dots<x_M<x_1+1$. Let $d_m$ be given by \eqref{definition:d_m}. Let $\tau_1,\dots,\tau_M$ be nonnegative real numbers. Let $K$ be a positive integer. We apply \eqref{positive-symmetric} with $N=KM$. For integers $k$ and $m$ with $1\le m\le M$, put $\lambda_{kM+m}=k+x_m$. Then $\delta_{kM+m}=d_m$. If $0\le k<K$, put $t_{kM+m}=\tau_m$. On inserting into \eqref{positive-symmetric}, we obtain
\begin{equation}\label{K-inequality}
2\sum_{m=1}^M\sum_{k=1}^{K-1}\frac{(K-k)d_m^2\tau_m^2}{k^2}+\sum_{m=1}^M\sum_{\substack{n=1\\n\neq m}}^M\sum_{\substack{k\in\mathbb{Z}\\\vert k\vert<K}}\frac{(K-\vert k\vert)d_m^\frac{3}{2}d_n^\frac{1}{2}\tau_m\tau_n}{\left(x_m-x_n-k\right)^2}\le C_3K\sum_{m=1}^M\tau_m^2.
\end{equation}
Now, since the series
$$
\sum_{k=1}^\infty\frac{1}{k^2}=\frac{\pi^2}{6}\quad\text{and}\quad\sum_{k\in\mathbb{Z}}\frac{1}{(x-k)^2}=\frac{\pi^2}{\sin^2(\pi x)}
$$
converges, it follows that they are (C, 1) summable to the same values, which is to say that
$$
\lim_{K\rightarrow\infty}\frac{1}{K}\sum_{k=1}^{K-1}\frac{K-k}{k^2}=\frac{\pi^2}{6}\quad\text{and}\quad\lim_{K\rightarrow\infty}\frac{1}{K}\sum_{\substack{k\in\mathbb{Z}\\\vert k\vert<K}}\frac{K-\vert k\vert}{(x-k)^2}=\frac{\pi^2}{\sin^2(\pi x)}.
$$
Hence, dividing \eqref{K-inequality} by $\pi^2K$ and letting $K\rightarrow\infty$ gives \eqref{trigonometric}.

($\Leftarrow$) Suppose that \eqref{trigonometric} holds. Let $\left(\lambda_k\right)_{k=-\infty}^\infty$ be a strictly increasing sequence of real numbers, and let $\delta_k\coloneqq\min\left\{\lambda_k-\lambda_{k-1},\lambda_{k+1}-\lambda_k\right\}$. Let $t_1,\dots,t_N$ be nonnegative real numbers. Let $0<\varepsilon<\frac{1}{2\left(\lambda_N-\lambda_0\right)}$. We apply \eqref{trigonometric} with $M=N$. For positive integers $n\le N$, put $x_n=\varepsilon\lambda_n$ and $\tau_n=t_n$. Then $d_n\ge\varepsilon\delta_n$, and \eqref{trigonometric} implies
$$
\frac{\varepsilon^2}{3}\sum_{n=1}^N\delta_n^2t_n^2+\sum_{m=1}^N\sum_{\substack{n=1\\n\neq m}}^N\frac{\varepsilon^2\delta_m^\frac{3}{2}\delta_n^\frac{1}{2}t_mt_n}{\sin^2\left(\pi\varepsilon\left(\lambda_m-\lambda_n\right)\right)}\le\frac{C_3}{\pi^2}\sum_{n=1}^Nt_n^2.
$$
On multiplying by $\pi^2$ and letting $\varepsilon\rightarrow0^+$, we obtain \eqref{positive-symmetric}.
\end{proof}

\begin{lemma}\label{lemma:L-sum-estimate}
For positive real numbers $B<1$ and positive integers $L$, we have
\begin{equation}\label{estimate1}
\sum_{\ell=1}^L\frac{L+1-\ell}{\sin^2\left(\frac{\pi\ell B}{L}\right)}=\frac{L^3}{6B^2}-\frac{L^2\log L}{\pi^2B^2}+O_B\left(L^2\right).
\end{equation}
\end{lemma}
\begin{proof}
From the identity $\frac{\pi^2}{\sin^2(\pi x)}=\sum_{k\in\mathbb{Z}}\frac{1}{(x-k)^2}$, we see that if $0<x\le B$, then
$$
\frac{\pi^2}{\sin^2(\pi x)}-\frac{1}{x^2}=\sum_{n=1}^\infty\left(\frac{1}{(n+x)^2}+\frac{1}{(n-x)^2}\right)<\sum_{n=1}^\infty\left(\frac{1}{n^2}+\frac{1}{(n-B)^2}\right).
$$
Hence, for $0<x\le B$, we have $\frac{1}{\sin^2(\pi x)}=\frac{1}{\pi^2x^2}+O_B(1)$.
Applying this estimate to each term on the left side of \eqref{estimate1}, we obtain
\begin{align*}
\sum_{\ell=1}^L\frac{L+1-\ell}{\sin^2\left(\frac{\pi\ell B}{L}\right)}&=\sum_{\ell=1}^L\frac{L^2(L+1-\ell)}{\pi^2\ell^2B^2}+O_B\left(\sum_{\ell=1}^L(L+1-\ell)\right)\\
&=\frac{L^2(L+1)}{\pi^2B^2}\sum_{\ell=1}^L\frac{1}{\ell^2}-\frac{L^2}{\pi^2B^2}\sum_{\ell=1}^L\frac{1}{\ell}+O_B\left(L^2\right).
\end{align*}
Since $\sum_{\ell=1}^L\frac{1}{\ell^2}=\frac{\pi^2}{6}+O\left(\frac{1}{L}\right)$ and $\sum_{\ell=1}^L\frac{1}{\ell}=\log L+O(1)$, the result \eqref{estimate1} follows.
\end{proof}
\begin{proof}[Proof~of~Theorem~\emph{\ref{main-negative}}]
To prove a lower bound for $\overline{C}\left(\frac{1}{2}\right)$, we apply \eqref{trigonometric} with particular sets of values. Let $K$ be a positive integer. Let $A$ and $B$ be positive real numbers such that $(K+1)A+B=1$. Let $L\ge\frac{B}{A}$ be an integer. We apply \eqref{trigonometric} with $M=K+L+1$. Choose $x_k=kA$ for $1\le k\le K$ and $x_{K+\ell+1}=(K+1)A+\frac{\ell B}{L}$ for $0\le\ell\le L$. Then $d_k=A$ for $1\le k\le K$ and $d_{K+\ell+1}=\frac{B}{L}$ for $0\le\ell\le L$. Choose $\tau_k=\frac{1}{\sqrt{K}}$ for $1\le k\le K$ and $\tau_{K+\ell+1}=\frac{u}{\sqrt{L+1}}$ for $0\le\ell\le L$ where $u$ is a nonnegative real number to be chosen later. Then \eqref{trigonometric} implies
\begin{align}\label{ABKLu-inequality}
&\frac{A^2}{3}+\frac{u^2B^2}{3L^2}+\frac{2A^2}{K}\sum_{k=1}^{K-1}\frac{K-k}{\sin^2(\pi kA)}+\frac{2u^2B^2}{L^2(L+1)}\sum_{\ell=1}^L\frac{L+1-\ell}{\sin^2\left(\frac{\pi\ell B}{L}\right)}\\
&+u\sqrt{\frac{AB}{KL(L+1)}}\left(A+\frac{B}{L}\right)\sum_{k=1}^K\sum_{\ell=0}^L\frac{1}{\sin^2\left(\pi\left(kA+\frac{\ell B}{L}\right)\right)}\le\frac{C_3}{\pi^2}\left(1+u^2\right).\nonumber
\end{align}
We observe that
\begin{align*}
\lim_{L\rightarrow\infty}\frac{1}{L}\sum_{k=1}^K\sum_{\ell=0}^L\frac{1}{\sin^2\left(\pi\left(kA+\frac{\ell B}{L}\right)\right)}&=\sum_{k=1}^K\int_0^1\frac{dx}{\sin^2(\pi(kA+Bx))}\\
&=\frac{1}{\pi B}\sum_{k=1}^K(\cot(\pi(K+1-k)A)+\cot(\pi kA))\\
&=\frac{2}{\pi B}\sum_{k=1}^K\cot(\pi kA).
\end{align*}
Now we let $L\rightarrow\infty$ in \eqref{ABKLu-inequality} and use the above estimate and Lemma \ref{lemma:L-sum-estimate}, obtaining
$$
\frac{A^2}{3}+\frac{2A^2}{K}\sum_{k=1}^{K-1}\frac{K-k}{\sin^2(\pi kA)}+\frac{u^2}{3}+\frac{2u}{\pi}\sqrt{\frac{A^3}{BK}}\sum_{k=1}^K\cot(\pi kA)\le\frac{C_3}{\pi^2}\left(1+u^2\right).
$$
That is,
\begin{equation}\label{equation:g(u)}
g(u)\coloneqq\frac{\kappa_0+\kappa_1u+\frac{u^2}{3}}{1+u^2}\le\frac{C_3}{\pi^2},
\end{equation}
where $\kappa_0$ and $\kappa_1$ depend on $A$, $B$, and $K$ and are given by
$$
\kappa_0\coloneqq\frac{A^2}{3}+\frac{2A^2}{K}\sum_{k=1}^{K-1}\frac{K-k}{\sin^2(\pi kA)}\quad\text{and}\quad\kappa_1\coloneqq\frac{2}{\pi}\sqrt{\frac{A^3}{BK}}\sum_{k=1}^K\cot(\pi kA).
$$
We find that $g(u)$ is maximized on $u\ge0$ at $u=u_0\coloneqq\frac{1}{\kappa_1}\left(\frac{1}{3}-\kappa_0+\sqrt{\left(\frac{1}{3}-\kappa_0\right)^2+\kappa_1^2}\right)$.
On inserting $u=u_0$ in \eqref{equation:g(u)}, we get
$$
G_K(A)\coloneqq\frac{1}{2}\left(\frac{1}{3}+\kappa_0+\sqrt{\left(\frac{1}{3}-\kappa_0\right)^2+\kappa_1^2}\right)\le\frac{C_3}{\pi^2}.
$$

\begin{figure}[h]
\includegraphics[scale=0.4]{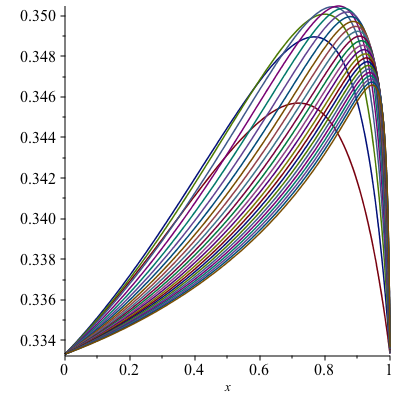}
\caption{The plot of $G_K\left(\frac{x}{K+1}\right)$ for $K=1,\dots,25$ and $0<x<1$.}
\label{G_Kgraph}
\end{figure}

Figure \ref{G_Kgraph} shows the plot of $G_K\left(\frac{x}{K+1}\right)$ for $K=1,\dots,25$ and $0<x<1$. We find
$$G_5(0.14)>0.35047.$$
By Lemma \ref{lemma:trigonometric-equivalence}, this gives the lower bound $\frac{C_3}{\pi^2}\ge0.35047$ for any absolute constant $C_3$ such that \eqref{positive-symmetric} holds. Since \eqref{positive-symmetric} holds with $C_3=\overline{C}\left(\frac{1}{2}\right)$, the result follows.
\end{proof}

\section*{Acknowledgements}

The author wishes to thank Professor Hugh Montgomery for suggesting the topic. The author would like to thank Professor Jeffrey Lagarias for helpful comments on the writing of this paper. This work was partially supported by NSF-grants DMS-1701576 and DMS-1701577.


\begin{thebibliography}{99}

\bibitem{GV1981}
Graham, S. W. and Vaaler, J. D. (1981). A class of extremal functions for the Fourier transform. \emph{Trans. Amer. Math. Soc.} \textbf{265}, 283--302.

\bibitem{HLP1952}
Hardy, G. H., Littlewood, J. E., and P\'{o}lya, G. (1952). \emph{Inequalities.} Second Edition. Cambridge. Cambridge University Press.

\bibitem{Li2005}
Li, X.-J. (2005). A note on the weighted Hilbert's inequality. \emph{Proc. Amer. Math. Soc.} \textbf{133}, 1165--1173.

\bibitem{Mon1978}
Montgomery, H. L. (1978). The analytic principle of the large sieve. \emph{Bull. Amer. Math. Soc.} \textbf{84}, 547--567.

\bibitem{Mon1994}
Montgomery, H. L. (1994). \emph{Ten Lectures on the Interface Between Analytic Number Theory and Harmonic Analysis.} CBMS Regional Conference Series in Mathematics \textbf{84}. Providence, RI. American Mathematical Society.

\bibitem{MV1999}
Montgomery, H. L. and Vaaler, J. D. (1999). A further generalization of Hilbert's inequality. \emph{Mathematika} \textbf{46}, 35--39.

\bibitem{MV1973}
Montgomery, H. L. and Vaughan, R. C. (1973). The large sieve. \emph{Mathematika} \textbf{20}, 119--134.

\bibitem{MV1974}
Montgomery, H. L. and Vaughan, R. C. (1974). Hilbert's inequality. \emph{J. Lond. Math. Soc. (2)} \textbf{8}, 73--82.

\bibitem{PP1991}
Pan, Chengdong and Pan, Chengbiao (1991). \emph{Fundamentals of Analytic Number Theory.} Fundamentals of Modern Mathematics Series (Chinese) \textbf{33}. Beijing. Science Press.

\bibitem{Pre1984}
Preissmann, E. (1984). Sur une in\'{e}galit\'{e} de Montgomery-Vaughan. \emph{Enseign. Math.} \textbf{30}, 95--113.

\bibitem{PL2013}
Preissmann, E. and L\'{e}v\^{e}que, O. (2013). On generalized weighted Hilbert matrices. \emph{Pacific J. Math.} \textbf{265}, 199--219.

\bibitem{Sch1911}
Schur, I. (1911). Bemerkungen zur Theorie der beschr\"{a}nkten Bilinearformen mit unendlich vielen Ver\"{a}nderlichen. \emph{Journal f. Math.} \textbf{140}, 1--28.

\bibitem{Sha1984}
Shan, Z. (1984). Hilbert's inequality. \emph{Kexue Tongbao} (Chinese) \textbf{29}, 62.

\bibitem{Wey1908}
Weyl, H. (1908). \emph{Singul\"{a}re Integralgleichungen mit besonderer Ber\"{u}cksichtigung des Fourierschen Integraltheorems}. Inaugural-Dissertation. G\"{o}ttingen.

\end{thebibliography}
\end{document}